\documentclass[a4paper,12pt]{article}
\usepackage[
    top=3cm,
    bottom=2.5cm,
    inner=3.0cm,
    outer=3.0cm,
    footskip=1.4cm,
    headsep=0.8cm,
    headheight=0.5cm,
]{geometry}
\usepackage[utf8]{inputenc}
\usepackage[T1]{fontenc}
\usepackage[english]{babel}
\usepackage{amsmath,amsfonts,amssymb,amsthm,amstext,amscd}
\usepackage{bbm}
\usepackage{microtype}
\usepackage{lmodern}
\usepackage[
	colorlinks=true,
	linkcolor=blue, 
	citecolor=red, 
	urlcolor=magenta, 
    hyperfootnotes=false
]{hyperref}
\usepackage{xcolor}
\usepackage{graphicx}
\usepackage{subcaption}
\usepackage{tikz}
\usepackage{dsfont}
\usepackage{fancyhdr}
\usepackage{enumerate}
\usepackage{authblk}
\usepackage{tikz-cd}
\usepackage[figurename=Fig.]{caption}

\usepackage[style=numeric,backend=biber]{biblatex}
\theoremstyle{definition}
\newtheorem{theorem}{Theorem}[section]
\newtheorem{lemma}[theorem]{Lemma}
\newtheorem{corl}[theorem]{Corollary}
\newtheorem{prop}[theorem]{Proposition}

\newtheorem{exemp}[theorem]{Example}

\newtheorem{defin}[theorem]{Definition}

\newcommand{\symmat}[1]{\mathbb{S}^{#1}}
\newcommand{\psdmat}[1]{\mathbb{S}_+^{#1}}
\newcommand{\nnegreal}[1]{\mathbb{R}_{+}^{#1}}
\newcommand{\innprod}[2]{\langle #1,#2 \rangle}
\newcommand{\tr}[1]{\operatorname{tr}(#1)}
\newcommand{\fld}[1]{\mathbb{#1}}
\newcommand{\allones}{\mathbbm{1}}
\newcommand{\matalgebra}[1]{M_{#1}(\mathbb{R})}
\newcommand{\littletaller}{\mathchoice{\vphantom{\big|}}{}{}{}}
\newcommand\restr[2]{{%
  \left.\kern-\nulldelimiterspace
  #1
  \littletaller
  \right|_{#2}
}}

\renewcommand{\sc}{\textsc}

\begin{document}
\pagestyle{plain}
\begin{center}
\textbf{\Large Total Conformal Rigidity in Graphs} \\
\vspace{1.5cc}
{ \sc Henrique Assumpção\footnotemark[1],
  Gabriel Coutinho\footnotemark[1],
  Chris Godsil\footnotemark[2]}\\
\vspace{0.3cm}
{\small \footnotemark[1]Department of Computer Science, Federal University of Minas Gerais, Brazil \\
\footnotemark[2]Department of Combinatorics \& Optimization, University of Waterloo, Canada}\\
\vspace{0.3cm}
{\small \texttt{[henrique.soares,gabriel]@dcc.ufmg.br}, \texttt{cgodsil@uwaterloo.ca}}
\end{center}

\begin{abstract}
We introduce and study a generalization of conformal rigidity for graphs.
A graph is \textit{$k$-conformally rigid} if the uniform edge weights simultaneously
maximize the sum of the $k$ smallest nontrivial Laplacian eigenvalues and minimize
the sum of the $k$ largest, over all normalized non-negative weight assignments. 
A graph that is $k$-conformally rigid for every $k$ is called \textit{totally conformally rigid}.
Our main result is a complete characterization: a graph is totally conformally
rigid if and only if it is \textit{edge-rigid}, meaning every canonical spectral
embedding onto a Laplacian eigenspace is edge-isometric. We further show this
is equivalent to all edges of the graph being pairwise Laplacian-cospectral, that is, the removal of any single edge yields a graph with the same Laplacian characteristic polynomial.
Using semidefinite programming duality, we establish this equivalence and
derive a polynomial-time algorithm for deciding edge-rigidity using only integer
arithmetic. We provide a combinatorial characterization of edge-rigidity in terms
of powers of the Laplacian matrix and connect it to the walk-regularity of signed line graphs.
We show that a graph is edge-rigid if and only if it is either $1$-walk-regular or
$1$-walk-biregular, and we finally show an equivalence based on monotone gauges and gauge duality. As an application, we derive two non-trivial combinatorial consequences of total conformal rigidity, relating it to the number of spanning trees and the Kirchhoff index of the graph.
\end{abstract}

\section{Introduction}\label{sec:intro}
Let $G = (V,E)$ be a simple, connected, undirected graph on $n$ vertices with at least one edge, and let $A := A(G)$ denote its adjacency matrix. 
For an edge $ab \in E$, we define
\[
z_{ab} := e_a - e_b,
\]
where $e_a \in \fld{R}^V$ denotes the $a$-th canonical basis vector, and we also let $L_{ab} := z_{ab}z_{ab}^T$.
Given a vector of edge weights $w \in \fld{R}^E$, the weighted Laplacian is the linear operator $\mathcal{L}:\fld{R}^E\rightarrow \symmat{V}$ such that
\begin{equation}\label{eq:laplacian_def}
    \mathcal{L}(w) := \sum_{ab \in E}w_{ab}L_{ab} = \sum_{ab \in E}w_{ab}z_{ab}z_{ab}^T,
\end{equation}
where $\symmat{V}$ denotes the space of symmetric matrices with rows and columns indexed by $V$. 
If we restrict ourselves to the set of nonnegative edge weights $\nnegreal{E}$, then the image of $\mathcal{L}$ is contained in the cone of positive semidefinite matrices $\psdmat{V}$. 
The adjoint $\mathcal{L}^*:\symmat{V} \mapsto \fld{R}^{E}$ of $L$ is given by
\[
\innprod{\mathcal{L}(w)}{X} = \innprod{w}{\mathcal{L}^*(X)}\quad\text{for any }w \in \fld{R}^E,\;
X \in \symmat{V},
\]
where $\innprod{A}{B} = \tr{AB^T}$ for any two matrices $A,B \in \matalgebra{V}$.
By expanding this definition, we obtain
\begin{equation}
    \mathcal{L}^*(X)_{ab} = z_{ab}^T X z_{ab} = X_{aa} + X_{bb} - 2X_{ab}
    \quad \text{for any } ab \in E.
\end{equation}
Throughout this work, we will be interested in optimization problems for the eigenvalues of $\mathcal{L}(w)$ with respect to normalized edge weights $w$. 
The simplex of valid edge weights is thus defined as
\[
\Delta_E := \{w \in \fld{R}^E: w \geq 0,\; \allones^Tw = |E|\}.
\]
For the remainder of the text, we shall use $\allones$ to denote the all-ones vector with the appropriate dimension, which will be clear from the context.
For instance, in the definition of $\Delta_E$ above, $\allones \in \fld{R}^E$.
We shall order the eigenvalues of $\mathcal{L}(w)$ as
\[
0 = \lambda_1(w) \leq \lambda_2(w) \leq \ldots \leq \lambda_n(w).
\]
For any $w \in \Delta_E$, it follows that $\lambda_1(w) = 0$, since $\mathcal{L}(w)\allones = 0$. 
The rank of $\mathcal{L}(w)$ is thus at most $n-1$, but it may be smaller since $w$ is allowed to have entries equal to zero. 
In the case of the unweighted Laplacian $L := \mathcal{L}(\allones)$, since $G$ is connected, we have that $\lambda_2(\allones) > \lambda_1(\allones) = 0$, and hence $L$ has rank $n-1$ (see, for instance,~\cite[Ch.13]{godsil2001algebraic}).
By the \textit{nontrivial} eigenvalues of $\mathcal{L}(w)$, we shall mean $\lambda_2(w),\ldots,\lambda_n(w)$.

In \cite{steinerberger2025conformallyrigidgraphs}, Steinerberger and Thomas studied
the concept of \textit{conformal rigidity} of a graph, which can be defined as
follows:
\begin{defin}
We say that a graph $G$ is \textit{conformally rigid} if for
any $w,w' \in \Delta_E$, we have
\[
\lambda_2(w) \leq \lambda_2(\allones) \quad\text{and}\quad \lambda_n(\allones) \leq \lambda_n(w'),
\]
that is, if $\allones$ maximizes the second smallest eigenvalue of $\mathcal{L}(w)$, and minimizes
its largest eigenvalue over all valid edge weights. If the former holds, we say the
graph is \textit{lower conformally rigid}, and if the latter holds, we say it is
\textit{upper conformally rigid}.
\end{defin}
The terminology is inspired by conformal geometry: just as a conformal map
distorts distances while preserving angles, a reweighting of edges distorts the
spectral geometry of the graph while leaving its combinatorial structure intact.
A conformally rigid graph resists such distortion --- no reweighting can push the
spectrum beyond its unweighted range. Conformal rigidity can be characterized in
terms of spectral embeddings onto the eigenspaces of $\lambda_2(\allones)$ and
$\lambda_n(\allones)$. The authors of~\cite{steinerberger2025conformallyrigidgraphs} also show that distance-regular graphs are conformally rigid, 
and provide sufficient conditions for Cayley graphs to exhibit conformal rigidity. 
In a follow-up paper, Gouveia, Steinerberger and Thomas~\cite{gouveia2025conformalrigidityspectralembeddings}
further studied necessary and sufficient conditions for Cayley graphs and
vertex-transitive graphs to be conformally rigid.

In this work, we generalize conformal rigidity to sums of eigenvalues.
More precisely, we define the following notion:
\begin{defin}\label{defin:k_conf_rigid}
    Let $S_k(w)$ and $s_k(w)$ denote the sum of the $k$ largest and $k$ smallest
    nontrivial eigenvalues of $\mathcal{L}(w)$, respectively, for some $k \in [n-1]$.
    We say that $G$ is \textit{$k$-conformally rigid} if, for any $w,w' \in \Delta_E$,
    \[
    s_k(w) \leq s_k(\allones) \quad\text{and}\quad S_k(\allones) \leq S_k(w').
    \]
    If $\allones$ minimizes $S_k$ over $\Delta_E$, we say that $G$
    is \textit{upper $k$-conformally rigid}, and
    if it maximizes $s_k$, we say it is \textit{lower $k$-conformally rigid}.
    If $G$ is $k$-conformally rigid for all $k \in [n-1]$, we say that $G$ is
    \textit{totally conformally rigid}.
\end{defin}
The notion of conformal rigidity introduced
in~\cite{steinerberger2025conformallyrigidgraphs} is equivalent to our notion of
$1$-conformal rigidity.

Laplacian eigenvalues arise naturally in several areas of graph theory
and combinatorial optimization. Ky Fan's classical result~\cite{kyfan}
characterizes sums of these eigenvalues variationally, and they appear prominently in the study
of isoperimetric constants~\cite{MOHAR1989274} and graph
partitioning~\cite{mohar1991laplacian}. More recently, sums of Laplacian
eigenvalues have been studied in connection with the \emph{spectral width}
$\lambda_n(w) - \lambda_2(w)$, whose minimization over edge weights is itself a
semidefinite program related to uniform sparsest cuts~\cite{steinerberger2025conformallyrigidgraphs}.
A well-known open problem involving these sums is \emph{Brouwer's conjecture}~\cite{Haemers2010-jk},
which asserts that, for any simple undirected graph $G$,
$S_k(\allones) \leq |E| + \binom{k+1}{2}$ for any $k \in [n-1]$.
The conjecture has been verified for many classes of graphs but remains open in general.

A key notion in our work is that of \textit{edge-rigidity}---the condition that every canonical spectral embedding of $G$ onto a Laplacian eigenspace is edge-isometric. As we will show, if $L = \sum_{i=1}^r\lambda_iE_i$ is the spectral decomposition of $L$, then edge-rigidity is equivalent to the condition that $\mathcal{L}^*(E_i) = \gamma_i\allones$ for some constant $\gamma_i > 0$ for every $i \in \{2,\ldots,r\}$. Therefore, a consequence of edge-rigidity is that $\mathcal{L}^* (L^\dagger)$ is a constant vector, where $L^\dagger$ is the Moore-Penrose pseudoinverse of $L$, leading to a vast list of derived properties satisfied by edge-rigid graphs, for instance:
\begin{itemize}
    \item The effective resistance of any edge is the same, and equal to $(|V|-1)/|E|$.
    \item Every edge has the same commute time from one of its endpoints to the other.
    \item The number of spanning trees containing any given edge is the same (that is, the graph is equiarboreal).
\end{itemize}
The fact that the last condition holds for 1-walk-regular graphs is due to Godsil~\cite{equiarboreal_godsil}, and is a consequence of the fact that $1$-walk regular graphs are edge-rigid, as we will see in Section~\ref{sec:walk_regularity}.

Our main contributions are as follows:
\begin{itemize}
    \item In Sections~\ref{sec:spec_embd} and~\ref{sec:cospectrality}, we introduce
    the notion of \emph{edge-rigidity} and show that
    it is equivalent to requiring that all edges of the graph be pairwise Laplacian-cospectral,
    meaning the removal of any single edge yields a graph with the same Laplacian
    characteristic polynomial.
    We further show that all edge-transitive graphs are edge-rigid.
    \item In Section~\ref{sec:sdp_eigensums}, by analyzing the optimal solutions to
    semidefinite programs (SDPs) used to model $k$-conformal rigidity, we show that
    total conformal rigidity is equivalent to edge-rigidity. We prove this by means of
    an auxiliary Lemma that provides further insight into upper $k$-conformal rigidity for
    certain values of $k$.
    \item In Section~\ref{sec:laplacian_walks}, we provide a combinatorial
    characterization of edge-rigidity in terms of powers of the Laplacian matrix, and show how to use this to obtain a polynomial-time algorithm for deciding edge-rigidity using
    only integer arithmetic. We also characterize edge-rigidity in terms of
    walk-regularity of signed line graphs.
    \item In Section~\ref{sec:walk_regularity}, we show that every edge-rigid graph
    is either regular or biregular bipartite, and that in each case the walk structure
    is highly constrained: regular edge-rigid graphs are walk-regular, and biregular
    bipartite edge-rigid graphs are walk-biregular. Combined with earlier results,
    this yields a complete characterization: a graph is edge-rigid if and only if it
    is $1$-walk-regular or $1$-walk-biregular.
    \item In Section~\ref{sec:gauge_duality}, we show that edge-rigidity is equivalent to an equality involving the monotone gauge function given by the sum of the eigenvalues of $\mathcal{L}(w)$ and its dual, providing a novel perspective on the algebraic and geometric properties of edge-rigid graphs.
    \item In Section~\ref{sec:majorization}, we derive two non-trivial combinatorial consequences of total conformal rigidity by applying the theory of majorization of vectors, relating conformal rigidity to the number of spanning trees and the Kirchhoff index of the graph.
\end{itemize}

\section{Edge-rigidity and spectral embeddings}\label{sec:spec_embd}

In this section, we will define the notion of edge-rigidity in terms of spectral
embeddings of graphs. These were originally introduced by Hall~\cite{hall1970}
to study the problem of placing $n$ connected points in Euclidean space, and have
since become a prominent technique with applications in many different areas of
computer science and mathematics, such as graph
clustering~\cite{LUO20032213,Gallagher02072024} and data
mining~\cite{spec_signed,spec_stat}.

Let $\lambda > 0$ be an eigenvalue of $L$ with multiplicity $m$, and let
$\mathcal{E}_\lambda \cong \fld{R}^m$ be its corresponding eigenspace.

\begin{defin}
Let $U \in \fld{R}^{n \times m}$ be a matrix whose columns form an orthonormal
basis for $\mathcal{E}_\lambda$. The collection of vectors
$\mathcal{U} = \{u_1, \ldots, u_n\} \subset \fld{R}^m$, where $u_i = U^Te_i$ is
the $i$-th row of $U$, is called the \textit{canonical spectral embedding} of $G$
onto $\mathcal{E}_\lambda$.
\end{defin}

These embeddings provide a geometric realization of $G$ in $\fld{R}^m$. Because
the columns of $U$ are eigenvectors corresponding to $\lambda > 0$, they are
orthogonal to $\allones$. Consequently, the embedding is centered at the origin,
meaning $\sum_{i=1}^n u_i = 0$. A particularly structured class of embeddings
consists of those where all adjacent vertices are separated by the same Euclidean
distance.

\begin{defin}
A canonical spectral embedding $\mathcal{U} = \{u_1,\ldots,u_n\}$ onto $\mathcal{E}_\lambda$ is \textit{edge-isometric} if there exists a constant
$\gamma > 0$ such that $\|u_a - u_b\|^2 = \gamma$ for all $ab \in E$.
\end{defin}

The existence of an edge-isometric embedding places severe algebraic constraints
on the graph. Let $X := U U^T$ be the orthogonal projector onto $\mathcal{E}_\lambda$.
Note that $X$ is independent of the choice of orthonormal basis $U$: different bases
yield geometrically congruent embeddings related by an orthogonal transformation of
$\fld{R}^m$, and the edge-isometric property is therefore an intrinsic property of
the eigenspace $\mathcal{E}_\lambda$ rather than of any particular basis. The squared
distance between adjacent vertices $a$ and $b$ is given by
\[
\|u_a - u_b\|^2 = \|u_a\|^2 + \|u_b\|^2 - 2u_a^Tu_b = X_{aa} + X_{bb} - 2X_{ab}
= \mathcal{L}^*(X)_{ab},
\]
hence the canonical embedding onto $\mathcal{E}_\lambda$ is edge-isometric if and
only if $\mathcal{L}^*(X) = \gamma\allones$ for some constant $\gamma > 0$.

The notion of conformal rigidity studied in~\cite{steinerberger2025conformallyrigidgraphs}
requires that the embeddings onto $\mathcal{E}_{\lambda_2}$ and $\mathcal{E}_{\lambda_n}$
alone be edge-isometric. We are interested in the stronger condition that
\emph{every} canonical embedding is edge-isometric, capturing a global spectral
regularity of the graph:
\begin{defin}
    We say that a graph $G$ is \textit{edge-rigid} if all canonical embeddings onto
    the Laplacian eigenspaces are edge-isometric, that is, for every eigenprojector
    $E_i$ of $L$ corresponding to a nonzero eigenvalue there exists $\gamma_i > 0$
    such that $\mathcal{L}^*(E_i) = \gamma_i\allones$.
\end{defin}
Before proceeding, we briefly remark on the embedding induced by the trivial eigenvalue $\lambda_1(\allones) = 0$.
Its eigenprojector is $E_1 := (1/n)\allones\allones^T$, and thus every vertex would be mapped into a single vector,
resulting in an embedding which is trivially edge-isometric, with $\mathcal{L}^*(E_1) = 0$. 
Since this is true for any connected graph, we omit it from the definitions above.

\section{Laplacian cospectrality}\label{sec:cospectrality}

Recall that two graphs $G$ and $H$ are called \textit{cospectral} if their adjacency matrices share the same characteristic
polynomial. The study of cospectral graphs was started in the 1950s~\cite{gunthard_old}, and has since become a central topic in spectral graph theory~\cite{Schwenk1973,Godsil1982,VANDAM2003241}.
Recently, Godsil, Sun and Zhang~\cite{godsil2025cospectralgraphsobtainededge} showed how to use $1$-walk-regularity to obtain families of non-isomorphic cospectral graphs via edge deletion.
In this section, we shall similarly define a notion of cospectrality for edges that can be used to characterize edge-rigidity. In a later section, we shall see that these notions can also be characterized in terms of $1$-walk-regular graphs. 

Given an edge $ab \in E$, the matrix $L - L_{ab}$ is the Laplacian of the graph $G - ab$ obtained by removing the edge $ab$ while retaining the vertices $a$ and $b$, so we define:

\begin{defin}
    Two edges $ab, cd \in E$ are called \textit{Laplacian-cospectral} if the
    matrices
    \[
    L - L_{ab}\quad\text{and}\quad L - L_{cd}
    \]
    have the same characteristic polynomial.
\end{defin}

Similarly to vertex cospectrality, the notion of
Laplacian-cospectrality can be stated in terms of the adjugate of the
characteristic polynomial of $L$. To do so, we will need the following standard
result, which we state without proof (see, for
instance,~\cite{horn1990matrix} or~\cite{bhatiaMatrixAnalysis1997}).

\begin{lemma}\label{lemma:matrix_det_lemma}
For a general square matrix $M$,
\[
\det(M + uv^T) = \det(M) + v^T\operatorname{adj}(M)u,
\]
where $\operatorname{adj}$ denotes the adjugate of $M$.
\hfill$\blacksquare$
\end{lemma}

\begin{lemma}\label{lemma:adj_lemma}
    Let $ab, cd \in E$. Then $ab$ and $cd$ are Laplacian-cospectral if and only if
    \[
    z^T_{ab}\operatorname{adj}(xI-L)z_{ab} = z^T_{cd}\operatorname{adj}(xI-L)z_{cd}
    \]
    as polynomials in $x$.
\end{lemma}
\begin{proof}
Take $M := xI - L$ and $u := v := z_{ab}$. Then
\[
    \det(xI-(L-L_{ab}))
    =
    \det(xI-L+z_{ab}z_{ab}^{T}),
\]
and, from Lemma~\ref{lemma:matrix_det_lemma}, we have
\[
    \det(xI-(L-L_{ab}))
    =
    \det(xI-L)
    +
    z_{ab}^{T}\operatorname{adj}(xI-L)z_{ab}.
\]
The same formula holds for the edge $cd$. Since the first term $\det(xI-L)$ is
independent of the edge, the characteristic polynomials of $L-L_{ab}$ and
$L-L_{cd}$ are equal if and only if
\[
    z_{ab}^{T}\operatorname{adj}(xI-L)z_{ab}
    =
    z_{cd}^{T}\operatorname{adj}(xI-L)z_{cd}. \qquad \qedhere
\]
\end{proof}

With this, we can prove an equivalence between edge-rigidity and
Laplacian-cospectrality:

\begin{theorem}\label{thm:cospectral_rigid_char}
    All edges of $G$ are pairwise Laplacian-cospectral if and only if $G$ is edge-rigid.
\end{theorem}
\begin{proof}
Let $L := \sum_{i=1}^r\lambda_iE_i$ be the spectral decomposition of $L$, with
$0 = \lambda_1 < \ldots < \lambda_r$, and let $m_i := \tr{E_i}$ be the multiplicity
of $\lambda_i$. The characteristic polynomial of $L$ is
\[
    p(x):=\det(xI-L)=\prod_{i=1}^{r}(x-\lambda_i)^{m_i}.
\]
In the $\lambda_i$-eigenspace, the matrix $\operatorname{adj}(xI-L)$ acts as
multiplication by
\[
    \frac{p(x)}{x-\lambda_i}
    =
    (x-\lambda_i)^{m_i-1}
    \prod_{j\neq i}(x-\lambda_j)^{m_j},
\]
therefore
\[
    \operatorname{adj}(xI-L)
    =
    \sum_{i=1}^{r}
    \frac{p(x)}{x-\lambda_i}E_i.
\]
From Lemma~\ref{lemma:adj_lemma}, all edges are Laplacian-cospectral if and only
if $z_{ab}^{T}\operatorname{adj}(xI-L)z_{ab}$ is independent of the edge $ab$.
Using the spectral decomposition of $\operatorname{adj}(xI-L)$, we have
\[
    z_{ab}^{T}\operatorname{adj}(xI-L)z_{ab}
    =
    \sum_{i=1}^{r}
    \frac{p(x)}{x-\lambda_i}
    z_{ab}^{T}E_i z_{ab}.
\]
We now claim that the polynomials $\{p(x)/(x-\lambda_i)\}_{i=1}^r$ are linearly
independent. Indeed, the polynomial $p(x)/(x-\lambda_i)$ vanishes at $x=\lambda_i$
with order $m_i-1$, whereas $p(x)/(x-\lambda_j)$ vanishes at $x=\lambda_i$ with
order $m_i$ for $j\neq i$. If a linear combination of said polynomials equals zero,
then taking the derivative of order $m_i-1$ at $x=\lambda_i$
isolates the $i$-th coefficient, which in turn implies that each coefficient has to be zero.
It then follows that the polynomial $z_{ab}^{T}\operatorname{adj}(xI-L)z_{ab}$
is independent of $ab$ if and only if each coefficient $z_{ab}^{T}E_i z_{ab}$ is
independent of $ab$. This is precisely the statement that $\mathcal{L}^*(E_i)$ is constant
for every $i$, which concludes the proof.
\end{proof}

The previous characterization allows us to prove edge-rigidity for edge-transitive
graphs:

\begin{corl}\label{corl:edge_transitive_rigid}
    All edge-transitive graphs are edge-rigid.
\end{corl}
\begin{proof}
Let $ab, cd\in E$. Since $G$ is edge-transitive, there is an automorphism $\sigma$
of $G$ mapping the edge $ab$ to the edge $cd$. Let $P$ be the permutation matrix
corresponding to $\sigma$. Then $P^{T}LP=L$, and
\[
    P^{T}z_{ab}=\pm z_{cd},
\]
hence
\[
    P^{T}L_{ab}P
    =
    P^{T}z_{ab}z_{ab}^{T}P 
    =
    z_{cd}z_{cd}^{T}
    =
    L_{cd}.
\]
Therefore $P^{T}(L-L_{ab})P = L-L_{cd}$, so $L-L_{ab}$ and $L-L_{cd}$ are
similar, and hence cospectral.
\end{proof}
\section{SDPs and $k$-Conformal Rigidity}\label{sec:sdp_eigensums}
In this section, we will show how to model $\max_w s_k(w)$ and $\min_w S_k(w)$ as semidefinite programs, and use the structure of their optimal solutions to study $k$-conformal rigidity. 

We start with the well-known formulations to compute sums of eigenvalues as semidefinite programs. These are straightforward consequences of the work of Ky Fan~\cite{kyfan}, and are now standard in the semidefinite programming literature (see, for instance,~\cite{Ben-Tal2001-hl,boyd2004convex}). Define the following sets of positive semidefinite matrices:
\begin{align}
      \mathcal{X}_k & := \{ X \in \symmat{V} : 0 \preceq X \preceq I,\;\tr{X} = k \} \label{Xk},\\
      \mathcal{Z}_k & := \{Z \in \mathcal{X}_k: Z\allones = 0\},
\end{align}
where for $A, B \in \symmat{V}$, we write $A \succeq B$ to denote that
$A-B \in \psdmat{V}$. It can be shown that
\begin{equation}\label{eq:max_k_eig_sdp}
  S_k(w) = \max_{X \in \mathcal{X}_k} \tr{\mathcal{L}(w)X},
\end{equation}
and, as $\mathcal{L}(w)$ has a zero eigenvalue with eigenvector $\allones$, it can also be shown that
\begin{equation}\label{eq:min_k_eig_sdp}
  s_k(w) = \min_{Z \in \mathcal{Z}_k} \tr{\mathcal{L}(w)Z}.
\end{equation}

If $\mathcal{L}(w) = \sum_{i=1}^n\lambda_iv_iv_i^T$ is the spectral decomposition of
$\mathcal{L}(w)$, with $\|v_i\| = 1$, then an optimal solution to~\eqref{eq:max_k_eig_sdp} is given by
\[X = \sum_{i=n-k+1}^nv_iv_i^T\]
and an optimal solution to~\eqref{eq:min_k_eig_sdp} is given by
\[Z = \sum_{i=2}^{k+1}v_iv_i^T.\]

If $L$ has $r$ distinct eigenvalues, and $k = k_j = \sum_{i=r-j+1}^r\tr{E_i}$ for
some $j \in [r]$, then $X = \sum_{i=n-k+1}^nv_iv_i^T$ is the unique optimizer
to~\eqref{eq:max_k_eig_sdp}. If $j \in [r-1]$ and $k = k_{j} + t$ with
$0 < t < \tr{E_{r-j}}$, then any matrix of the form $X = X_j + \tilde{E}$,
where $E_{r-j} \succeq \tilde{E} \succeq 0$ and $\tr{\tilde{E}} = t$, is an
optimal solution, and analogous results hold for the $k$ smallest non-trivial eigenvalues. In fact, we continue our treatment below focusing on $S_k$, but all results have analogous versions for $s_k$.

Consider the problem of minimizing $S_k(w)$ over all valid edge weights:
\[
\min_{w \in \Delta_E}S_k(w) = \min_{w \in \Delta_E}\max_{X \in \mathcal{X}_k}
\tr{\mathcal{L}(w)X}.
\]
By writing the semidefinite dual of~\eqref{eq:max_k_eig_sdp}, we can express this
as
\begin{equation}\label{eq:min_topeigsum_primal}
    \min_{w \in \Delta_E}S_k(w) = \min\{ky + \tr{Y} :
    yI + Y - \mathcal{L}(w) \succeq 0,\; Y \succeq 0,\; y \in \fld{R},\; w\in\Delta_E\}.
\end{equation}
The corresponding dual is
\begin{equation}\label{eq:min_topeigsum_dual}
    \max\{|E|x : \mathcal{L}^*(X) \geq x\allones,\; X \in \mathcal{X}_k,\; x \in \fld{R}\}.
\end{equation}
Note that setting $(Y=I,y=\lambda_r+1,w=\allones)$ and $(X=(k/n)I,x=0)$ yields Slater points for both primal and dual programs,
hence strong duality holds (see, for instance, \cite{boyd2004convex}). By explicitly writing the complementary slackness conditions, we obtain:

\begin{lemma}\label{lemma:compl_slack}
    If $(Y,y,w)$ and $(X,x)$ are feasible solutions for~\eqref{eq:min_topeigsum_primal}
    and~\eqref{eq:min_topeigsum_dual}, respectively, then they are optimal if and
    only if
    \[
    X(Y + yI - \mathcal{L}(w)) = 0,\quad XY = Y,\quad \innprod{\mathcal{L}^*(X) - x\allones}{w} = 0.
    \]
\end{lemma}
\begin{proof}
    We consider the following chain of inequalities:
\[
\begin{split}
|E|x
& = \langle x \allones,w\rangle \\
&\leq \langle \mathcal{L}^*(X),w\rangle \\
&= \langle X,\mathcal{L}(w)\rangle \\
&\leq \langle X,Y+yI\rangle \\
&= y\tr{X} + \innprod{X}{Y}\\
&\leq ky+ \innprod{I}{Y}.
\end{split}
\]
As strong duality holds, at optimal solutions equality must hold throughout the chain. Hence
\[
\innprod{\mathcal{L}^*(X) - x\allones}{w} = 0,\quad \innprod{X}{Y + yI - \mathcal{L}(w)} = 0,\quad
\innprod{I-X}{Y} = 0.
\]
For the latter two equalities, since both matrices in each inner product are positive semidefinite,
their product must be zero, which yields the desired conditions. 
The converse is immediate.
\end{proof}

These conditions allow us to prove the following:

\begin{lemma}\label{lemma:necessity_lemma}
    Suppose $G$ is a graph with unweighted Laplacian $L = \sum_{i=1}^r\lambda_iE_i$,
    with $0 = \lambda_1 < \ldots < \lambda_r$. Fix $j \in [r-1]$, and let
    $k_j = \sum_{i=r-j+1}^r\tr{E_i}$ and $X_j = \sum_{i=r-j+1}^rE_i$.
    Then the following hold:
    \begin{enumerate}[(1)]
        \item If $G$ is upper $k_j$-conformally rigid, then $\mathcal{L}^*(X_j) = \beta\allones$
        for some $\beta > 0$.

        \item If $j \in [r-2]$ and $G$ is both upper $k_{j}$-conformally rigid and
        upper $k_{j+1}$-conformally rigid, then the canonical embedding onto the
        $(r-j)$-th eigenspace of $L$ is edge-isometric. Moreover, $G$ is also upper
        $k$-conformally rigid for all intermediate $k_{j} < k < k_{j+1}$.
    \end{enumerate}
    Analogous versions of these results hold for lower conformal rigidity.
\end{lemma}
\begin{proof}
    For~(1), assume that $G$ is upper $k_j$-conformally rigid, so $w = \allones$
    attains the minimum of $S_{k_j}(w)$ over $\Delta_E$. Combining this with
    the definition of~\eqref{eq:min_topeigsum_dual}, we get that if $(X,x)$ is
    optimal, then
    \[
    S_{k_j}(\allones) = \min_{w \in \Delta_E}S_{k_j}(w) = |E|x = \innprod{x\allones}{\allones},
    \]
    and as Lemma~\ref{lemma:compl_slack} gives $\mathcal{L}^*(X) = x\allones$, this implies
    $S_{k_j}(\allones) = \innprod{L}{X}$. From the observations at the beginning
    of this section, the unique optimal matrix attaining $S_{k_j}(\allones)$
    in~\eqref{eq:max_k_eig_sdp} is exactly $X_j = \sum_{i=r-j+1}^r E_i$, thus
    $X = X_j$, and setting $\beta = x$ gives $\mathcal{L}^*(X_j) = \beta\allones$, as desired.

    For~(2), assume $G$ is upper $k_{j}$-conformally rigid and upper $k_{j+1}$-conformally rigid.
    The complementary slackness conditions of Lemma~\ref{lemma:compl_slack} combined with the fact that $w = \allones > 0$
    guarantees the existence of dual solutions $(X_{j}, x_{j})$ and $(X_{j+1}, x_{j+1})$
    for~\eqref{eq:min_topeigsum_dual} such that:
    \[
    \begin{split}
        \mathcal{L}^*(X_{j}) &= x_{j}\allones, \quad \text{with } |E|x_{j} = S_{k_{j}}(\allones),\\
        \mathcal{L}^*(X_{j+1}) &= x_{j+1}\allones, \quad \text{with }
        |E|x_{j+1} = S_{k_{j+1}}(\allones).
    \end{split}
    \]
    We note that $x_j = \innprod{L}{X_j}/|E|$, and since $E_{r-j} = X_{j+1} - X_j$,
    setting $\gamma = x_{j+1} - x_j > 0$ gives $\mathcal{L}^*(E_{r-j}) = \gamma\allones$,
    which proves the first claim. Now let $k$ be an integer with $k_{j} < k < k_{j+1}$.
    We can write $k = k_j + \alpha\tr{E_{r-j}}$, where
    $\alpha = (k - k_j)/\tr{E_{r-j}} \in (0,1)$, and define
    \[
    X_{\alpha} = \sum_{i=r-j+1}^rE_i + \alpha E_{r-j}\quad\text{and}\quad
    x_\alpha = \frac{\innprod{L}{X_\alpha}}{|E|}.
    \]
    One can check that $(X_\alpha, x_\alpha)$ is feasible for~\eqref{eq:min_topeigsum_dual}.
    Since this dual program is a maximization problem, we have for any
    $w \in \Delta_E$:
    \[
    S_k(w) \geq |E|x_\alpha = S_{k_j}(\allones) + \alpha\tr{E_{r-j}}\lambda_{r-j}
    = S_k(\allones),
    \]
    hence $S_k(w) \geq S_k(\allones)$ for all $w \in \Delta_E$, proving that
    $G$ is upper $k$-conformally rigid.
\end{proof}

With this auxiliary Lemma, we can prove the main result of this section:
\begin{theorem}\label{thm:total_rigid_char}
    A graph $G$ is totally conformally rigid if and only if it is edge-rigid.
\end{theorem}
\begin{proof}
    Let $0 = \lambda_1 < \ldots < \lambda_r$ be the distinct eigenvalues of $L$,
    fix $j \in [r-1]$, and let $k_j = \sum_{i=r-j+1}^r\tr{E_i}$ and
    $X_j = \sum_{i=r-j+1}^rE_i$. For the forward direction, we apply
    Lemma~\ref{lemma:necessity_lemma} for each $k_j$ with $j \in [r-1]$ to
    conclude edge-rigidity.

    For the backward direction, assume $\mathcal{L}^*(E_i) = \gamma_i\allones$ with
    $\gamma_i > 0$. Define
    \[
    x_j = \sum_{i=r-j+1}^r\gamma_i,\quad Y_j = \sum_{i=r-j+1}^{r}(\lambda_i -
    \lambda_{r-j})E_i,\quad y_j = \lambda_{r-j}.
    \]
    One can easily verify that $(X_j, x_j)$ and $(Y_j, y_j, \allones)$ satisfy the
    conditions of Lemma~\ref{lemma:compl_slack}, and thus form a pair of
    primal-dual optimal solutions. Since their objective value is $S_{k_j}(\allones)$,
    this implies that $G$ is upper $k_j$-conformally rigid for all $j \in [r-1]$.
    Combining this with Lemma~\ref{lemma:necessity_lemma} allows us to conclude
    that $G$ is $k$-conformally rigid for all $k \in [n-1]$, implying that $G$ is
    totally conformally rigid.
\end{proof}

Combining Theorems~\ref{thm:cospectral_rigid_char} and~\ref{thm:total_rigid_char}
gives us a four-way equivalence:

\begin{corl} \label{corl:equiv_rigid_cospectral}
    For a graph $G$, the following are equivalent:
    \begin{enumerate}[(a)]
        \item All edges of $G$ are pairwise Laplacian-cospectral.
        \item $G$ is edge-rigid.
        \item $G$ is totally conformally rigid.
        \item For every Laplacian eigenprojector $E_i$, the vector $\mathcal{L}^*(E_i)$ is
        constant.
    \end{enumerate}
    \hfill$\blacksquare$
\end{corl}

We conclude this section with an observation: in the proofs above, we only used
upper $k$-conformal rigidity for every $k$, and in fact it is immediate to notice that a graph is upper $k$-conformally rigid if and only if it is lower $(n-1-k)$-conformally rigid,
for $k \in [n-2]$.

\begin{prop}
A graph $G$ is upper $k$-conformally rigid for all $k \in [n-1]$ if and only if
it is lower $k$-conformally rigid for all $k \in [n-1]$.
\hfill$\blacksquare$
\end{prop}

\section{Laplacian Powers and Line Graphs}\label{sec:laplacian_walks}

In this section, we give a combinatorial characterization of total conformal
rigidity in terms of powers of the Laplacian matrix.

\begin{theorem}\label{thm:total_walk_char}
A graph $G$ is edge-rigid if and only if for every integer $\ell \ge 0$, there
exists a constant $C_\ell \in \fld{R}$ such that
\[
\mathcal{L}^*(L^\ell)_{ab} = L^\ell_{aa} + L^\ell_{bb} - 2L^\ell_{ab} = C_\ell
\quad \text{for every } ab \in E.
\]
\end{theorem}
\begin{proof}
For the forward direction, assume $G$ is edge-rigid. The orthogonal projectors
onto the eigenspaces of $L$ satisfy $\mathcal{L}^*(E_i) = \gamma_i \allones$ for some
$\gamma_i > 0$ and each $i \in \{2,\ldots,r\}$. Using the spectral decomposition
$L = \sum_{i=1}^r \lambda_i E_i$, we can write
\[
  L^\ell = \sum_{i=1}^r \lambda_i^\ell E_i.
\]
Applying the linear operator $\mathcal{L}^*$ yields
\[
  \mathcal{L}^*(L^\ell) = \sum_{i=1}^r \lambda_i^\ell \mathcal{L}^*(E_i)
  = \left( \sum_{i=1}^r \lambda_i^\ell \gamma_i \right) \allones.
\]
Noting that $\lambda_1 = 0$ and setting $C_\ell := \sum_{i=2}^r \lambda_i^\ell \gamma_i$ gives the desired
implication.

For the reverse direction, assume $\mathcal{L}^*(L^\ell) = C_\ell \allones$ for every
integer $\ell \ge 0$. By applying $\mathcal{L}^*$ to the spectral decomposition, we obtain
a linear system for $\ell \in \{0, 1, \ldots, r-1\}$:
\[
  \sum_{i=1}^r \lambda_i^\ell \mathcal{L}^*(E_i) = C_\ell \allones.
\]
Because the $r$ eigenvalues $\lambda_1, \ldots, \lambda_r$ are distinct,
the coefficient matrix of this system is an $r \times r$ Vandermonde matrix, which
is invertible. Consequently, each vector $\mathcal{L}^*(E_i) \in \nnegreal{E}$ can be
uniquely expressed as a linear combination of the constant vectors
$\{C_0\allones, C_1\allones, \ldots, C_{r-1}\allones\}$. This implies that
$\mathcal{L}^*(E_i) = \gamma_i\allones$ for some $\gamma_i \in \fld{R}$ and every
$i \in \{2,\ldots,r\}$, hence $G$ is edge-rigid.
\end{proof}

By noting that a walk-regular graph must also have $L^\ell$ with constant diagonals for all $\ell \geq 0$,
and also that such Laplacian powers must also be constant on the entries corresponding to edges for $1$-walk-regular graphs,
we can conclude the following:
\begin{corl}\label{corl:walk_reg_edge_rigid_1wr}
    If $G$ is walk-regular, then it is edge-rigid if and only if it is
    $1$-walk-regular.
    
    \hfill$\blacksquare$
\end{corl}

\begin{exemp}
    There are edge-rigid graphs that are not edge-transitive.
    Indeed, every strongly regular graph is $1$-walk-regular, and by the Corollary
    above every strongly regular graph is edge-rigid.
    Consequently, any strongly regular graph that is not edge-transitive gives an
    example.
    For instance, the Chang graphs~\cite{Chang1959} on 28 vertices are strongly
    regular with parameters $(28,12,6,4)$ and are not edge-transitive.
\end{exemp}

\begin{figure}[htbp]
     \centering
     \includegraphics[width=\textwidth]{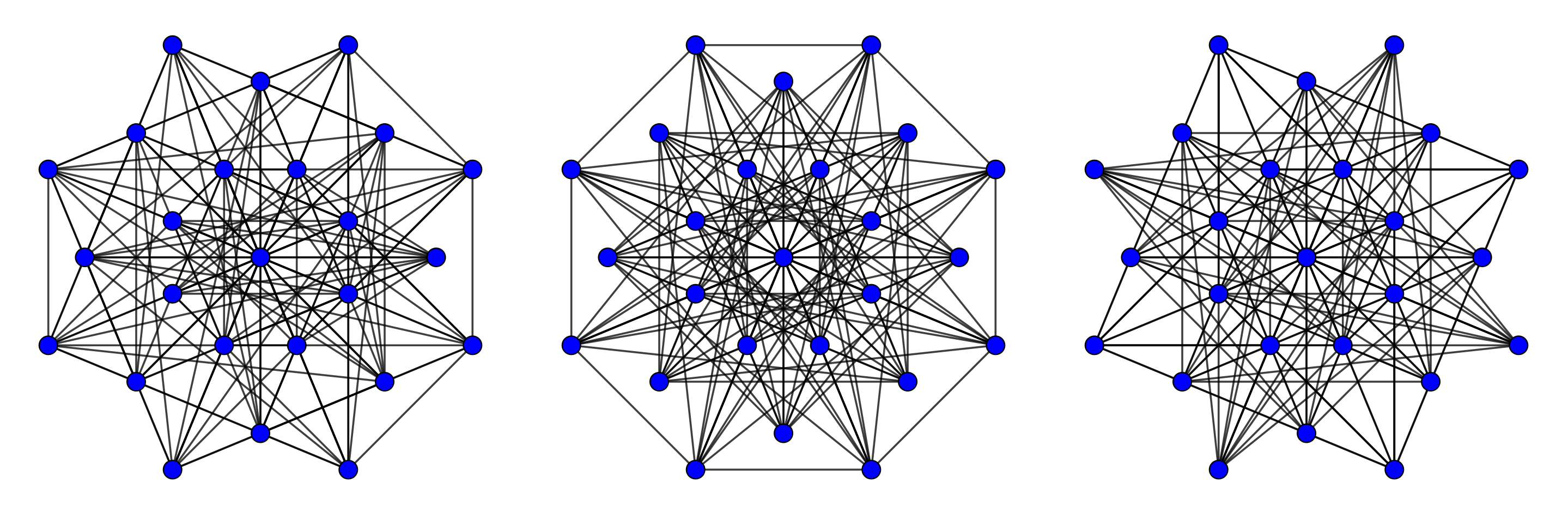}
     \caption{The three Chang graphs~\cite{Chang1959}, which are edge-rigid but
     not edge-transitive.}
     \label{fig:chang}
\end{figure}

In~\cite{steinerberger2025conformallyrigidgraphs}, the authors utilize semidefinite programming to numerically test and certify wether a graph is conformally rigid. While semidefinite programs can be solved to within any prescribed tolerance in polynomial time~\cite{Nesterov1994-fr}, this continuous optimization approach only provides an approximate numerical verification rather than an exact combinatorial decision procedure, and in practice it is susceptible to numerical instability and precision artifacts on larger instances.

An immediate and significant consequence of Theorem~\ref{thm:total_walk_char} is that total conformal rigidity can be decided exactly in deterministic polynomial time via a straightforward algebraic algorithm that requires only integer arithmetic. This is achieved by computing $\mathcal{L}^*(L^\ell)$ for all $\ell \in [n-1]$ and checking whether each resulting vector is constant. Because the eigenvalues of the unweighted Laplacian $L$ are bounded above by $2n$, the entries of $L^\ell$ for $\ell \in [n-1]$ are integers bounded in absolute value by $(2n)^{n-1}$. Consequently, the arithmetic computations involve integers whose bit-length grows at most linearly with $n \log n$. Since computing the sequence of matrix powers $L, L^2, \ldots, L^{n-1}$ requires $O(n^4)$ elementary arithmetic operations, the entire procedure can be implemented exactly without numerical error in polynomial time with respect to the bit-complexity model.

\begin{corl}\label{corl:poly_time}
    If $G=(V,E)$ is a graph on $n$ vertices, then there is a deterministic polynomial-time algorithm for deciding if $G$ is edge-rigid using $O(n^4)$ arithmetic operations over integers of bit-length $O(n \log n)$.
    \hfill$\blacksquare$
\end{corl}

The structural constraints imposed by edge-rigidity can be translated into the
language of line graphs. To make this precise, fix an arbitrary orientation of the
edges of $G$, and construct the oriented incidence matrix $B \in \fld{R}^{V \times E}$
such that the column corresponding to an edge $e = \{u,v\}$ directed from $v$ to
$u$ is $z_e = e_u - e_v$. While the Laplacian $L = BB^T$ is independent of this
orientation, the product $B^T B \in \fld{R}^{E \times E}$ depends on the choice of orientation.
The diagonal entries of $B^T B$ are $z_e^T z_e = 2$. For distinct
edges $e$ and $f$, the entry $(B^T B)_{ef}$ is $\pm 1$ if they share a vertex,
and $0$ otherwise. Thus, we may write $B^T B = 2I + A_\sigma$, where $A_\sigma$
is the adjacency matrix of a signed line graph.

The exact signature $\sigma$ depends on the chosen orientation of $G$. However,
reversing the orientation of an edge $e \in E$ simply negates the column $z_e$
in $B$. This operation conjugates $B^T B$ by a diagonal matrix $D$ with
$D_{ee} = -1$ and $D_{ff} = 1$ for all $f \neq e$. In the context of signed
graphs, this corresponds exactly to a vertex switching in the signed line graph.
Since switching-equivalent signed graphs have similar adjacency matrices, the
diagonal entries of $(A_\sigma^\ell)_{ee}$ are invariant under the choice of
orientation for every $\ell \geq 0$. We can thus prove the following:

\begin{corl}
A graph $G$ is edge-rigid if and only if, for any orientation of its edges, the
corresponding signed line graph of $G$ is walk-regular. \label{corl:line_walk}
\end{corl}
\begin{proof}
By Theorem~\ref{thm:total_walk_char}, $G$ is edge-rigid if and only if for every
integer $\ell \ge 0$, $\mathcal{L}^*(L^\ell)_{ab} = C_\ell$ for all $ab \in E$. Using the
factorization $L = BB^T$, we rewrite the adjoint operator acting on $L^\ell$ for
a specific edge $e = \{a,b\} \in E$. Let $e_{ab} \in \fld{R}^E$ denote the
standard basis vector for this edge, so that the corresponding column in the
oriented incidence matrix is $Be_{ab} = z_{ab}$. Then:
\[
    \mathcal{L}^*(L^\ell)_{ab} = z_{ab}^T (BB^T)^\ell z_{ab}
    = (Be_{ab})^T (BB^T)^\ell (Be_{ab})
    = e_{ab}^T (B^T B)^{\ell+1} e_{ab}.
\]
This final expression is exactly the $ab$-diagonal entry of the matrix
$(B^TB)^{\ell+1}$. Therefore, $\mathcal{L}^*(L^\ell)_{ab}$ is a global constant $C_\ell$
for all edges $ab \in E$ and all integers $\ell \ge 0$ if and only if the diagonal
of $(B^T B)^m$ is constant for all $m \ge 1$. Because $B^T B = 2I + A_{\sigma}$,
this is equivalent to the diagonal entries of $A_{\sigma}^m$ being constant for
all $m \ge 0$. Since this invariance holds regardless of the initial orientation
chosen for $G$, this is exactly the definition of the corresponding line graph being walk-regular.
\end{proof}

\section{Walk-Regularity and Biregularity}\label{sec:walk_regularity}

In this section, we further explore the connection between edge-rigidity and
walk-regularity: we show that all regular edge-rigid graphs are
walk-regular, and provide analogous results for bipartite biregular edge-rigid
graphs.

We first establish a combinatorial consequence of edge-rigidity:

\begin{lemma}\label{lemma:edge_rigid_reg_bireg}
    If $G$ is a connected edge-rigid graph, then the sum of the degrees of the endpoints of any edge is a global constant across the graph. Consequently, $G$ must be either regular or biregular bipartite.
\end{lemma}
\begin{proof}
By Theorem~\ref{thm:cospectral_rigid_char},
the orthogonal projectors $E_i$ onto the eigenspaces of the unweighted Laplacian
$L$ satisfy $\mathcal{L}^*(E_i) = \gamma_i\allones$ for scalars $\gamma_i > 0$. Writing
$L = \sum_{i=1}^r \lambda_i E_i$, it follows that $\mathcal{L}^*(L) = C\allones$ for some
$C \in \fld{R}$. For any edge $ab \in E$,
\[
      \mathcal{L}^*(L)_{ab} = L_{aa} + L_{bb} - 2L_{ab} = d_a + d_b + 2 = C,
\]
where $d_a$ denotes the degree of vertex $a$, so $d_a + d_b$ is a global constant
across all edges.

Now consider any two vertices $a$ and $c$ connected by a walk of length two
through an intermediate vertex $b$. We have $d_a + d_b = d_b + d_c$, which implies
$d_a = d_c$. If $a$ and $c$ are in turn connected by a walk of length $2\ell$, for some $\ell \geq 1$,
we can use the previous case as the base of an induction to easily conclude that any two vertices connected by walks of even length must share the same degree.
If the graph contains an odd cycle, then any two vertices are connected by a walk of even length, because the graph is connected and the odd cycle may be used to choose the parity of a walk connecting any two vertices. So all vertices have the same degree and $G$ is regular. If the graph contains no odd cycles, it is bipartite, and the two sides of the bipartition form two degree classes, making $G$ biregular bipartite.
\end{proof}

Thus every edge-rigid graph is either regular or biregular bipartite. We handle
the regular case first:

\begin{theorem}
    If $G$ is a regular edge-rigid graph, then $G$ is walk-regular.
\end{theorem}
\begin{proof}
    Assume $G$ is a $d$-regular edge-rigid graph with adjacency matrix $A$. Let
    $W_\ell(a) = (A^\ell)_{aa}$ denote the number of closed walks of length $\ell$
    starting at vertex $a \in V$, and let $W_\ell \in \fld{R}^V$ be the vector of
    these walk counts.

    Because $G$ is $d$-regular, its unweighted Laplacian is $L = dI - A$.
    By Theorem~\ref{thm:total_walk_char}, edge-rigidity implies that for any
    edge $ab \in E$, $\mathcal{L}^*(L^\ell)_{ab} = C_\ell$ for some constant $C_\ell$ and
    all $\ell \ge 0$. Since $A^\ell$ can be expanded as a polynomial in $L$,
    linearity implies that $\mathcal{L}^*(A^\ell)$ is also a constant vector. That is, there
    exist constants $C'_\ell$ such that for any edge $ab \in E$,
    \[
        \mathcal{L}^*(A^\ell)_{ab} = W_\ell(a) + W_\ell(b) - 2(A^\ell)_{ab} = C'_\ell.
    \]
    Summing over all $d$ neighbors $b$ adjacent to $a$:
    \[
        \sum_{b \sim a} W_\ell(a) + \sum_{b \sim a} W_\ell(b) -
        2 \sum_{b \sim a} (A^\ell)_{ab} = d C'_\ell.
    \]
    The first term equals $dW_\ell(a)$ and the second equals $(AW_\ell)_a$.
    For the third term, we observe that
    \[
    \sum_{b \sim a}(A^\ell)_{ab} = \sum_{b \in V}A_{ab}(A^\ell)_{ba}
    = (A^{\ell+1})_{aa} = W_{\ell+1}(a).
    \]
    Combining these, we obtain the recurrence
    \[
    dW_\ell + AW_\ell - 2W_{\ell+1} = dC'_\ell\allones.
    \]
    We now proceed by induction on $\ell$ to show that $W_\ell$ is a constant
    vector for all $\ell \ge 0$. The base case $\ell=0$ is trivial since
    $W_0 = \allones$. Assume $W_\ell = \alpha_\ell \allones$ for some scalar
    $\alpha_\ell$. Because $G$ is $d$-regular,
    $AW_\ell = A(\alpha_\ell\allones) = d\alpha_\ell\allones$.
    Substituting into the recurrence yields
    \[
        2W_{\ell+1} = d(2\alpha_\ell - C'_\ell)\allones,
    \]
    so $W_{\ell+1}$ is also a scalar multiple of $\allones$. By induction, the
    number of closed walks of any length is constant across all vertices, hence
    $G$ is walk-regular.
\end{proof}

Combining the previous result with Corollary~\ref{corl:walk_reg_edge_rigid_1wr} gives us a combinatorial characterization of regular edge-rigid graphs:

\begin{corl}
    If $G$ is a regular graph, then it is edge-rigid if and only if it is
    $1$-walk-regular.
    \hfill$\blacksquare$
\end{corl}

We now turn to edge-rigid graphs that are biregular bipartite. First, we introduce an analogous notion of walk-regularity for this case:

\begin{defin}
Let $G = (V_1 \cup V_2, E)$ be a bipartite graph. We say $G$ is
\textit{walk-biregular} if, for every integer $\ell \ge 0$, there exist constants
$\alpha_\ell^{(1)}$ and $\alpha_\ell^{(2)}$ such that the number of closed walks
of length $\ell$ starting at any vertex $a \in V_1$ is exactly $\alpha_\ell^{(1)}$,
and starting at any vertex $b \in V_2$ is exactly $\alpha_\ell^{(2)}$.
Furthermore, we say $G$ is \textit{$1$-walk-biregular} if it is walk-biregular
and the number of walks of length $\ell$ between the endpoints of any edge
$ab \in E$ is a global constant $\beta_\ell$ depending only on $\ell$.
\end{defin}

Note that because $G$ is bipartite, $\alpha_\ell^{(1)} = \alpha_\ell^{(2)} = 0$
for all odd $\ell$, and $\beta_\ell = 0$ for all even $\ell$. We can now state
and prove the bipartite counterpart to the regularity result above.

\begin{theorem}
If $G = (V_1 \cup V_2, E)$ is a biregular bipartite edge-rigid graph, then $G$
is walk-biregular.
\end{theorem}
\begin{proof}
Assume $G$ is connected and biregular bipartite with vertex partitions $V_1$ and
$V_2$, where vertices in $V_1$ have degree $d_1$ and vertices in $V_2$ have degree
$d_2$. We express the degree matrix as $D = d_1 I_1 + d_2 I_2$, where $I_1$ and
$I_2$ are the diagonal indicator matrices for $V_1$ and $V_2$, respectively.

Let $U_\ell(a) = (L^\ell)_{aa}$ and $U_\ell \in \fld{R}^V$. We will show by induction that $U_\ell$ is constant on
$V_1$ and on $V_2$ for every $\ell \geq 0$. The base case of $\ell = 0$ is immediate, since
$L^0 = I$. Assume the claim holds for all $k \leq \ell$. By Theorem~\ref{thm:total_walk_char}, $\mathcal{L}^*(L^\ell)_{ab} = C_\ell$
for every edge $ab$ with $a \in V_1$ and $b \in V_2$:
\begin{equation}\label{eq:biregular_walk}
    U_\ell(a) + U_\ell(b) - 2(L^\ell)_{ab} = C_\ell.
\end{equation}
By the inductive hypothesis, $U_\ell(a) = \gamma_\ell^{(1)}$ for all $a \in V_1$ and
$U_\ell(b) = \gamma_\ell^{(2)}$ for all $b \in V_2$, so Equation~\eqref{eq:biregular_walk}
shows that $(L^\ell)_{ab}$ is a global constant $\eta_\ell$ for all edges
$ab \in E$.

Now consider the diagonal of $L^{\ell+1} = LL^\ell = (D-A)L^\ell$. For any
$a \in V_1$:
\[
    U_{\ell+1}(a) = d_1(L^\ell)_{aa} - \sum_{b \sim a}(L^\ell)_{ab}
    = d_1\gamma_\ell^{(1)} - d_1\eta_\ell = d_1(\gamma_\ell^{(1)} - \eta_\ell),
\]
which is constant for all $a \in V_1$. A symmetric argument gives
$U_{\ell+1}(b) = d_2(\gamma_\ell^{(2)} - \eta_\ell)$, constant for all $b \in V_2$.
By induction, the diagonal of $L^\ell$ is constant on each partition for all
$\ell \ge 0$.

We must now show that this implies $G$ is walk-biregular.
As the entries of $A^\ell$ count the number of walks of length $\ell$ between vertices, it follows that for any $a \in V_1$,
$A^\ell e_a$ is entirely supported on $V_1$ for even $\ell$, and on $V_2$ for odd $\ell$. 
Consequently, $DA^\ell e_a = d_1 A^\ell e_a$ for even $\ell$, and $DA^\ell e_a  = d_2 A^\ell e_a$ for odd $\ell$.
Because $L = D - A$, the vector $L^\ell e_a = (D - A)^\ell e_a$ can be expanded as a polynomial in $D$ and $A$. By evaluating from right to left, every instance of $D$ acts as either $d_1$ or $d_2$ depending on the parity of the number of $A$'s that have already been applied. Therefore, $L^\ell e_a$ simplifies to a linear combination $\sum_{i=0}^\ell c_{i,\ell}^{(1)} A^i e_a$, where the coefficients depend only on $d_1, d_2,$ and $\ell$. 
By a symmetric argument, $(D-L)^\ell e_a$ also evaluates to a linear combination of $L^i e_a$.
Thus, $(A^\ell)_{aa} = e_a^T A^\ell e_a$ is a linear combination of the entries $(L^i)_{aa}$ for $i \le \ell$. Since we established by induction that $(L^i)_{aa}$ is constant across $V_1$ (and similarly for $V_2$), it follows that the diagonal entries of $A^\ell$ are constant on $V_1$ and $V_2$, proving that $G$ is walk-biregular.
\end{proof}

Analogously to the regular case, we can conclude the following:

\begin{corl}
    Let $G$ be a walk-biregular graph. Then $G$ is edge-rigid if and only if it
    is $1$-walk-biregular.
\end{corl}
\begin{proof}
Suppose $G$ is $1$-walk-biregular. Then for any $\ell \ge 0$, the diagonal entries of $A^\ell$ are constant on $V_1$ and $V_2$, and the off-diagonal entries $(A^\ell)_{ab}$ are a global constant across all edges $ab \in E$. Since $L^\ell e_a = (D-A)^\ell e_a$ can be expressed as $\sum_{i=0}^\ell c_{i,\ell}^{(1)} A^i e_a$ (and similarly for $e_b$), the entry $(L^\ell)_{ab} = e_a^T L^\ell e_b$ is a linear combination of the edge entries $(A^i)_{ab}$. Because the terms $(A^i)_{ab}$ are constant across all edges, $(L^\ell)_{ab}$ is also constant for all $ab \in E$. The same applies to the diagonals $(L^\ell)_{aa}$ and $(L^\ell)_{bb}$. Consequently, $\mathcal{L}^*(L^\ell)_{ab} = (L^\ell)_{aa} + (L^\ell)_{bb} - 2(L^\ell)_{ab}$ evaluates to a global constant $C_\ell$ for all $\ell \ge 0$. By Theorem~\ref{thm:total_walk_char}, $G$ is edge-rigid.

 Now suppose $G$ is edge-rigid. By Theorem~\ref{thm:total_walk_char}, $\mathcal{L}^*(L^\ell)_{ab}$ is a global constant $C_\ell$ for all $\ell \ge 0$. Since $G$ is walk-biregular, we know that the diagonal entries of $A^\ell$ are constant on $V_1$ and $V_2$, which by an analogous argument to the one above implies that the diagonal entries of $L^\ell$ must also be constant on $V_1$ and $V_2$. Since $\mathcal{L}^*(L^\ell)_{ab}$ and the diagonals are constant, $(L^\ell)_{ab}$ must be constant across all edges $ab \in E$. Finally, expanding $A^\ell e_a = (D-L)^\ell e_a$ as a linear combination of $L^i e_a$ demonstrates that $(A^\ell)_{ab}$ is a linear combination of the constant edge entries $(L^i)_{ab}$. Therefore, $(A^\ell)_{ab}$ is a global constant across all edges, meaning $G$ is $1$-walk-biregular.
\end{proof}

Combining the two cases yields a complete characterization of edge-rigidity:

\begin{corl} \label{corl:walk_reg_bireg}
A graph $G$ is edge-rigid if and only if it is either $1$-walk-regular or
$1$-walk-biregular.
\hfill$\blacksquare$
\end{corl}

\section{Gauges} \label{sec:gauge_duality}

Yet another perspective on edge-rigidity can be obtained by considering the notion of a gauge. The standard reference in a more general context is Rockafellar~\cite{rockafellar1997convex}, and a comprehensive treatment in the context of graph theory can be found at \cite{proenca2021dual}.

A function $g: \fld{R}^E_+ \to \fld{R}$ is called a \textit{gauge} if 
\begin{enumerate}[(a)]
\item $g$ is nonnegative, meaning that $g(w) \ge 0$ for all $w \in \fld{R}^E_+$, and $g(0) = 0$;
\item $g$ is positively homogeneous, meaning that $g(\alpha w) = \alpha g(w)$ for all $\alpha \ge 0$ and $w \in \fld{R}^E_+$;
\item $g$ is subadditive, meaning that $g(w + w') \le g(w) + g(w')$ for all $w, w' \in \fld{R}^E_+$.
\end{enumerate}
If in addition $g(w) = 0$ implies $w = 0$, then $g$ is positive definite, and if $g(w) \le g(w')$ whenever $w \le w'$ coordinate-wise, then we say the gauge is \textit{monotone}. Positive definite monotone gauges are restrictions of norms to the nonnegative orthant.

We recall the following minimization expression for $S_k(w)$ obtained by SDP strong duality from Eq.~\ref{eq:max_k_eig_sdp} in Section~\ref{sec:sdp_eigensums}:
\begin{equation}\label{eq:sum_topk_min}
    S_k(w) = \min\{ky + \tr{Y} :
    yI + Y - \mathcal{L}(w) \succeq 0,\; Y \succeq 0,\; y \in \fld{R}\}.
\end{equation}
With this, we can show the following:

\begin{lemma}
    The function $S_k:\fld{R}^E_+ \mapsto \fld{R}$ is a positive definite monotone gauge.
\end{lemma}
\begin{proof}
    Positive (semi)definiteness is immediate, as $\mathcal{L}(w) \succeq 0$ for all $w \in \fld{R}^E_+$, and $S_k(w) = 0$ implies that $\mathcal{L}(w) = 0$, which in turn implies $w = 0$. Positive homogeneity follows from the definition of $S_k$ and the fact that scaling $w$ by $\alpha$ scales $\mathcal{L}(w)$ by $\alpha$.
    To check subadditivity, let $w,w' \in \fld{R}^E_+$ and let $(Y,y)$ and $(Y',y')$ be the optimal solutions of the minimization SDPs for $S_k(w)$ and $S_k(w')$ (defined in~\eqref{eq:sum_topk_min}), respectively. Then $(Y + Y', y + y')$ is also a feasible solution for tha same SDP for $S_k(w + w')$ with objective value $S_k(w) + S_k(w')$, so $S_k(w + w') \le S_k(w) + S_k(w')$.
    Finally, monotonicity follows from the fact that if $w \le w'$, then $\mathcal{L}(w') - \mathcal{L}(w) = \mathcal{L}(w'- w) \succeq 0$, so any feasible solution for the maximization SDP of $S_k(w)$ (defined in~\eqref{eq:max_k_eig_sdp}) is also feasible for the SDP of $S_k(w')$, with an objective value at least that of the latter.
\end{proof}

Recall from \eqref{Xk} the definition $\mathcal{X}_k = \{ X \in \symmat{V} : 0 \preceq X \preceq I,\;\tr{X} = k \}$, and consider its image under $\mathcal{L}^*$:
\[ \mathcal{L}^*(\mathcal{X}_k) = \{ \mathcal{L}^*(X) : X \in \mathcal{X}_k \}. \]
The unit convex corner associated to $S_k^\circ$ is the downward hull of $\mathcal{L}^*(\mathcal{X}_k)$, and the unit convex corner associated to $S_k$ is also semidefinite representable.
\begin{prop}
    Consider the gauge $S_k$, and let $S_k^\circ$ be its dual gauge. Let $B_k$ be the unit convex corner associated to $S_k$, and let $B_k^\circ$ be the convex corner associated to $S_k^\circ$. Then these sets are given by:
    \begin{align}
        B_k^\circ & = \left\{ w \in\mathbb{R}_{+}^{E}: \exists \, y \in \mathcal{L}^*(\mathcal{X}_k) \text{ such that } w \leq y \right\}, \\
        B_k & = \left\{	w\geq 0: \exists\,t\in\mathbb{R},\ Z\succeq 0 \text{ such that } tI+Z \succeq \mathcal{L}(w),\ kt+ \tr{Z}\leq 1 \right\}.
    \end{align}
\end{prop}
\begin{proof}
    We briefly comment on this proof as it is a straightforward application of the definitions and standard duality theory. The first equality follows from monotonicity and the fact that
    \[S_k(w) = \max_{X \in \mathcal{X}_k} \langle \mathcal{L}(w) , X \rangle = \max_{X \in \mathcal{X}_k} w^T \mathcal{L}^*(X) = \max_{y \in \mathcal{L}^*(\mathcal{X}_k)} w^T y.\]
    The second equality follows from the minimization SDP formulation and the Minkowski functional presentation of the gauge $S_k$.
\end{proof}

The result below follows from a standard argument of gauge minimization in a simplex, but we include it here with its elementary proof in order to highlight the connection between the gauges $S_k$ and $S_k^\circ$ and the notion of upper $k$-conformal rigidity.

\begin{theorem} \label{thm:gauge_char}
    A graph $G$ is upper $k$-conformally rigid if and only if
    \[ S_k(\allones)S_k^\circ(\allones) = |E|. \]
\end{theorem}
\begin{proof}
    Recall that $\Delta_E := \{w \in \fld{R}^E: w \geq 0,\; \allones^Tw = |E|\}$. We start this proof with the following chain of equalities:
\begin{align}
    \min_{w \in \Delta_E} S_k(w) 
    & = \min_{w \in \Delta_E} \max_{y \in B_k^\circ} w^T y \label{eq1}\\
    & = \max_{y \in B_k^\circ} \min_{w \in \Delta_E} w^T y \label{eq2}\\
    & = \max_{y \in B_k^\circ} \left\{ |E| \min_{e \in E} y_e \right\}\label{eq3}\\
    & = |E| \max_{\lambda \geq 0} \{\lambda : \lambda \allones \in B_k^\circ \} \label{eq4}\\
    & = |E| \frac{1}{S_k^\circ(\allones)}. \label{eq5}
\end{align}
    The first equality follows from the fact that $S_k$ is the dual of $S_k^\circ$. The second equality follows from the minimax theorem, since the function $w^T y$ is linear in both $w$ and $y$, and the sets $\Delta_E$ and $B_k^\circ$ are convex and compact. The third equality follows from the fact that $\min_{w \in \Delta_E} w^T y$, for a fixed $y$, is achieved at a vertex of $\Delta_E$, which corresponds to putting all weight on a single edge. The fourth equality follows from the fact that $B_k^\circ$ is lower comprehensive, as $y \in B_k^\circ$ and $\lambda \leq \min_{e \in E} y_e$ implies $\lambda \allones \in B_k^\circ$. Finally, the last equality follows from the definition of the dual gauge $S_k^\circ$, since $S_k^\circ(\allones) = \min_{\lambda} \{ \lambda : \allones \in \lambda B_k^\circ \}$.

    The equivalence follows directly: $G$ is upper $k$-conformally rigid if and only if $S_k(\allones) = \min_{w \in \Delta_E} S_k(w)$, and by \eqref{eq5} this is equivalent to
    \[S_k(\allones) = \frac{|E|}{S_k^\circ(\allones)}.\]
\end{proof}

\section{Majorization and Combinatorial Consequences}\label{sec:majorization}

In this short section, we show two non-trivial combinatorial consequences of total conformal rigidity, obtained upon applying the theory of majorization of vectors (see Bhatia~\cite[Chapter II]{bhatiaMatrixAnalysis1997} for a standard reference). Recall that given two vectors $x,y \in \mathbb{R}^n$ ordered such that $x_1 \geq \ldots \geq x_n$ and $y_1 \geq \ldots \geq y_n$, we say that $x$ is majorized by $y$ if
\[
\sum_{i=1}^k x_i \leq \sum_{i=1}^k y_i \quad \text{for all } k \in [n] \quad \text{and} \quad \sum_{i=1}^n x_i = \sum_{i=1}^n y_i.
\]

Let $\operatorname{supp}(w) := \{e \in E : w_e > 0\}$ denote the support of a weight assignment $w \in \Delta_E$. Because edge weights are permitted to be zero, a choice of weights may result in a weighted graph that is disconnected even though the underlying simple graph $G$ is connected. We thus define the set of non-disconnecting valid edge weights as
\[
\Delta_E^+ := \{w \in \Delta_E : (V, \operatorname{supp}(w)) \text{ is connected}\},
\]
which includes all strictly positive weight assignments.
Equivalently, $\Delta_E^+$ consists of all weight vectors whose support contains the edge set of at least one spanning tree of $G$.
This combinatorial characterization is precisely what guarantees that the algebraic connectivity of the weighted graph remains strictly positive. That is, for any $w \in \Delta_E$, we have $w \in \Delta_E^+$ if, and only if,
\[
0 = \lambda_1(w) < \lambda_2(w) \leq \ldots \leq \lambda_n(w),
\]
meaning the weighted Laplacian $\mathcal{L}(w)$ maintains rank $n-1$. The results derived in this section rely on the following key observation: $G$ being totally conformally rigid is precisely equivalent to the vector of eigenvalues of $L$ being majorized by the vector of eigenvalues of $\mathcal{L}(w)$ for any $w \in \Delta_E^+$.

The weighted matrix tree theorem~\cite{kirchhoff1847ueber} states that for a given graph $G$ with edge weights $w \in \fld{R}^E$,
\[\tau(G,w) := \sum_T \prod_{e \in T} w_e = \frac{1}{n} \prod_{i = 2}^{n} \lambda_i (w),
\]
where the sum is over all spanning trees $T$ of $G$. We use this to prove the following:

\begin{theorem}
    If $G$ is totally conformally rigid, then $w = \allones$ maximizes $\tau(G,w)$ over all $w \in \Delta_E$.
\end{theorem}
\begin{proof}
    As $G$ is totally conformally rigid, it follows that for any $w \in \Delta_E^+$, the eigenvalues of $L$ are majorized by the eigenvalues of $\mathcal{L}(w)$. Since $\log$ is a strictly concave function on $\fld{R}_{++}$, we can apply~\cite[Theorem II.3.1]{bhatiaMatrixAnalysis1997} to conclude
    \[\sum_{i=2}^n \log \lambda_i(\allones) \ge \sum_{i=2}^n \log \lambda_i(w).\]
    Taking exponentials on both sides and applying the matrix tree theorem yields $\tau(G,\allones) \ge \tau(G,w)$ for all $w \in \Delta_E^+$. For weights $w \in \Delta_E \setminus \Delta_E^+$, the graph is disconnected, yielding $\tau(G,w) = 0 \le \tau(G,\allones)$. Thus, the maximum over the entire simplex is achieved at $w = \allones$.
\end{proof}

The Kirchhoff index of a weighted graph $G$ is defined as
\[ \mathrm{Kf}(G,w) := \sum_{\{a,b\} \subseteq V} z_{ab}^T \mathcal{L}(w)^\dagger z_{ab}, \]
where $\mathcal{L}(w)^\dagger$ is the Moore-Penrose pseudoinverse of the weighted Laplacian $\mathcal{L}(w)$ and $z_{ab} = e_a - e_b$. The term $z_{ab}^T \mathcal{L}(w)^\dagger z_{ab}$ is the effective resistance between vertices $a$ and $b$ (see \cite{KleinRandic1993ResistanceDistance} for a standard reference). If the graph is disconnected, we conventionally define $\mathrm{Kf}(G,w) = \infty$.

\begin{theorem}
    If $G$ is totally conformally rigid, then $w = \allones$ minimizes $\mathrm{Kf}(G,w)$ over all $w \in \Delta_E$.
\end{theorem}
\begin{proof}
    For $w \in \Delta_E \setminus \Delta_E^+$, the graph is disconnected and $\mathrm{Kf}(G,w) = \infty$, which is trivially suboptimal. For $w \in \Delta_E^+$, one can express the Kirchhoff index as
    \[\mathrm{Kf}(G,w) = n \sum_{i=2}^n \frac{1}{\lambda_i(w)},\]
    which follows from an elementary expansion of the definition (see~\cite{Gutman1996}).
    As $G$ is totally conformally rigid, the eigenvalues of $L$ are majorized by the eigenvalues of $\mathcal{L}(w)$ for any $w \in \Delta_E^+$. Since $x \mapsto 1/x$ is a convex function on $\fld{R}_{++}$, again applying~\cite[Theorem II.3.1]{bhatiaMatrixAnalysis1997} yields
    \[\sum_{i=2}^n \frac{1}{\lambda_i(\allones)} \le \sum_{i=2}^n \frac{1}{\lambda_i(w)},\]
    giving $\mathrm{Kf}(G,\allones) \le \mathrm{Kf}(G,w)$, which completes the proof.
\end{proof}
\section{Conclusion}\label{sec:conclusion}

We introduced and studied $k$-conformal rigidity, a generalization of the conformal
rigidity of Steinerberger and Thomas~\cite{steinerberger2025conformallyrigidgraphs},
and showed that the strongest form of this property --- total conformal rigidity ---
admits a complete characterization.

Our main result establishes the equivalence of seven conditions, following from Corollaries~\ref{corl:equiv_rigid_cospectral}, \ref{corl:line_walk}, \ref{corl:walk_reg_bireg}, and Theorem~\ref{thm:gauge_char}. These conditions range from spectral properties of the Laplacian, to combinatorial properties of walks, to geometric properties of gauges, to optimization properties of semidefinite programs.

\begin{corl}
    For a graph $G$, the following are equivalent:
    \begin{enumerate}[(a)]
        \item All edges of $G$ are pairwise Laplacian-cospectral.
        \item $G$ is edge-rigid.
        \item $G$ is totally conformally rigid.
        \item For every Laplacian eigenprojector $E_i$, the vector $\mathcal{L}^*(E_i)$ is
        constant.
        \item For any orientation of the edges of $G$, the corresponding signed line graph of $G$ is walk-regular.
        \item $G$ is either $1$-walk-regular or $1$-walk-biregular.
        \item For all $k \in [n-1]$
            \[ S_k(\allones)S_k^\circ(\allones) = |E|. \]
    \end{enumerate}
    \hfill$\blacksquare$
\end{corl}

Our main open question is to understand the structure of graphs that are $k$-conformally rigid for some specific values of $k$, but not for all $k$ simultaneously. Given any constant $C$, can we generate a graph that is $k$-conformally rigid for all $k \leq C$ but not for $C+1$? Are there graphs that are $2$-conformally rigid but not $1$-conformally rigid? Is there a constant $C < 1$ such that if a graph is $k$-conformally rigid for all $k \leq Cn$, then it is totally conformally rigid?

We would also like to understand which are the non totally conformally rigid graphs (if any) for which $w = \allones$ is the optimizer of $\tau(G,w)$ or $\mathrm{Kf}(G,w)$ over $w \in \Delta_E$.

\section*{Acknowledgments}
Henrique Assumpção and Gabriel Coutinho are supported by FAPEMIG and CNPq. 
Chris Godsil is supported by NSERC (Grant No. RGPIN-9439).

\printbibliography
\end{document}